\documentclass[12pt,reqno]{amsart}
\usepackage{a4wide}

\numberwithin{equation}{section}

\newtheorem{theorem}{Theorem}[section]

\newtheorem{lemma}[theorem]{Lemma}

\theoremstyle{remark}

\newcommand{\ep}{\varepsilon}
\newcommand{\vp}{\varphi}
\newcommand{\R}{\mathbb R^N_+}
\newcommand{\BR}{\partial \mathbb R_+^N}

\newcommand{\lp}{-\Delta}

\newcommand{\into}{\int_{\Omega}}
\newcommand{\intpo}{\int_{\partial \Omega}}

\newcommand{\intr}{\int_{\mathbb R_+^{N}}}
\newcommand{\intrr}{\int_{\mathbb R^{N-1}}}

\begin{document}
\title{Existence results for superlinear elliptic equations with nonlinear boundary value
conditions}
\author{Xiaohui Yu$^{*}$}
\thanks{*The Center for China's Overseas Interests, Shenzhen University,
Shenzhen Guangdong, 518060, People's Republic of
China(yuxiao\_211@163.com)}
 \maketitle

{\scriptsize  {\bf Abstract:} In this paper, we study the existence
of solutions for the following superlinear elliptic equation with
nonlinear boundary value condition
$$
 \left\{
\begin{array}{ll}
-\Delta u+u=|u|^{r-2}u &\text{in} \; \Omega,\\
\\ \frac{\partial u}{\partial \nu}=|u|^{q-2}u  &\text{on}\;\partial\Omega,
\end{array}
\right.
$$
where $\Omega\subset \mathbb R^N, N\geq 3$ is a bounded domain with
smooth boundary. We will prove the existence results for the above
equation under four different cases: (i) Both $q$ and $r$ are
subcritical; (ii) $r$ is critical and $q$ is subcritical; (iii) $r$
is subcritical and $q$ is critical; (iv) Both $q$ and $r$ are
critical.

{\bf keywords:}\;\;existence result, critical exponent, trace inequality, nonlinear boundary value, (PS) condition.\\

{\bf Mathematics Subject Classification (2010):}\;\; 35J60, 35J57,
35J15.
%\end {abstract}
\vspace{3mm} } \maketitle

\section { Introduction}
In this paper,  we study the existence of solutions for the
following superlinear elliptic equation with nonlinear boundary
value condition
\begin{equation}\label{1.1}
 \left\{
\begin{array}{ll}
-\Delta u+u=|u|^{r-2}u &\text{in} \; \Omega,\\
\\ \frac{\partial u}{\partial \nu}=|u|^{q-2}u  &\text{on}\;\partial\Omega,
\end{array}
\right.
\end{equation}
where $\Omega\subset \mathbb R^N, N\geq 3$ is a bounded domain with
smooth boundary. We also assume that $r,q>2$ so that this problem is
a superlinear one.

Nonlinear boundary value problems were widely studied in the past
few decades and there are many results on this aspect. For example,
Y.Li and M.Zhu \cite{LZu} classified all the positive solutions for
equation
\begin{equation}\label{1.2}
 \left\{
\begin{array}{ll}
-\Delta u=|u|^{2^*-2}u &\text{in} \; \R,\\
\\ \frac{\partial u}{\partial \nu}=|u|^{2_*-2}u  &\text{on}\;\partial\R,
\end{array}
\right.
\end{equation}
where $2^*=\frac{2N}{N-2}$ is the usual critical Sobolev exponent
and $2_*=\frac{2(N-1)}{N-2}$ is the critical exponent for Sobolev
trace inequality. They proved the positive solutions of problem
\eqref{1.2} must have the following form
\begin{equation}\label{1.3}
u_\ep(x)=\frac{[N(N-2)]^{\frac{N-2}{4}}\ep^{\frac{N-2}2}}{[\ep^2+|x'|^2+(x_N+\ep
x_N^0)^2]^{\frac{N-2}2}},
\end{equation}
where $x'=(x_1,x_2,\cdots,x_{N-1})$, $x_N^0=(\frac{N}{N-2})^{\frac
12}$ is a positive constant depending only on $N$. The main method
in this paper is the moving sphere method. Similar results can be
found in \cite{CSF}. Later, by using the moving sphere method and
Harnack inequality, Y.Li and L.Zhang \cite{LZ} studied the Liouville
type theorems for elliptic equations with nonlinear boundary value
conditions. X.Yu \cite{Yu6} studied the nonexistence results for
nonlinear boundary value problems with general nonlinearities. B.Hu
\cite{Hu} studied the nonexistence results of harmonic function with
nonlinear boundary condition. Other results can be found in
\cite{HY1}\cite{HY2}\cite{HY3}.

On the other hand, existence results for nonlinear boundary value
problems were also widely studied. For example, M.Chipota,
M.Chlebik, M.Fila and I.Shafrir \cite{CCFS} studied the existence
results for nonlinear boundary value problem in $\mathbb R_+^N$.
J.Bonder and J.Rossi \cite{BR} studied the existence results for
nonlinear boundary value problems involving $p-$Laplacian operator.
In order to overcome the difficulty of nonlocal property of the
fractional Laplacian operators, X.Cabre and J.Tan \cite{CT}
transformed the fractional Laplacian equations into nonlinear
boundary value problems by the extension theorem in \cite{CS}. Then
they studied the existence results, regularity results, nonexistence
results for fractional Laplacian equations. Later, J.Tan \cite{T}
studied the existence result for critical fractional Laplacian
equations. He still turned the problem into a nonlinear boundary
value problem. Other results on fractional Laplacian equations can
be found in \cite{BCP}\cite{BCPa}\cite{HYu} and etc.

In this paper, we study the existence results for nonlinear boundary
value problem \eqref{1.1}. This equation involves two nonlinear
terms. We first consider the case where both $r$ and $q$ are
subcritical. We have the following multiple solutions result.
\begin{theorem}\label{t 1.1}
Suppose that the nonlinear terms are superlinear and subcritical,
i.e., $2<r<\frac{2N}{N-2}$ and $2<q<\frac{2(N-1)}{N-2}$, then
problem \eqref{1.1} has infinitely many solutions.
\end{theorem}

However, if $r$ or(and) $q$ is critical, things become more
difficult. The main difficulty of solving the problem by critical
point theory lies in that the Sobolev embedding or(and) the Sobolev
trace embedding is not compact, then the so called Palais-Smale
condition is generally not satisfied by the related functional $I$.
The most significant achievement in this aspect is the work
\cite{BN} in which the authors studied the existence result of the
following problem
\begin{equation}\label{1.4}
 \left\{
\begin{array}{ll}
-\Delta u=|u|^{2^*-2}u+\lambda u &\text{in} \; \Omega,\\
\\ u=0 &\text{on}\;\partial\Omega.
\end{array}
\right.
\end{equation}
They first proved the corresponding functional satisfies the
$(PS)_c$ condition for $c\in (0,\frac 1NS^{\frac N2})$, where $S$ is
the best constant of $D^{1,2}(\mathbb R^N)\hookrightarrow
L^{2^*}(\mathbb R^N)$. Then they proved the Mountain Pass level of
the corresponding functional indeed belongs to this interval under
some assumptions on $\lambda$. Then the existence result was
obtained. After the work \cite{BN}, there were plenty of works on
critical Laplacian equations and we can't list all of them. For
example, X.Wang \cite{W} studied the existence results for critical
Neumann boundary value problem. J.Tan \cite{T} studied the existence
result for critical fractional Laplacian equations in bounded
domain. B.Barrios, E.Colorado, A.de Pablo and U.S\'{a}nchez
\cite{BCP} studied the existence result for a more general critical
fractional Laplacian equation. Hua and Yu \cite{HYu} studied the
existence result for critical fractional Laplacian equations under
other circumstance. Ye and Yu \cite{YY} studied the global
compactness results for critical Laplacian equations in the whole
spaces.

Inspired by the above works, we next consider the case where $q$ is
subcritical and $r$ is the critical exponent. We have the following
existence result.
\begin{theorem}\label{t 1.2}
Suppose that $r=\frac{2N}{N-2}$ and $2<q<\frac{2(N-1)}{N-2}$, then
problem \eqref{1.1} has at least one nontrivial solution.
\end{theorem}

Similarly, we can study the case where $r$ is subcritical and $q$ is
critical. For this case, we have the following existence result.
\begin{theorem}\label{t 1.3}
Suppose that $2<r<\frac{2N}{N-2}$ and $q=\frac{2(N-1)}{N-2}$, then
problem \eqref{1.1} has at least one nontrivial solution.
\end{theorem}

Finally, we study the most difficult case in which both $q$ and $r$
are critical. We have the following existence result.
\begin{theorem}\label{t 1.4}
Suppose that $r=\frac{2N}{N-2}$ and $q=\frac{2(N-1)}{N-2}$, then
problem \eqref{1.1} possesses at least one nontrivial solution.
\end{theorem}

The proof of the above theorems is based on the variational methods.
Obiviously, the solutions of problem \eqref{1.1} correspond to the
critical points of the functional
\begin{equation}\label{1.5}
I(u)=\frac 12\into |\nabla u|^2+u^2\,dx-\frac 1r\into
|u|^r\,dx-\frac 1{q}\intpo |u|^q\,dS
\end{equation}
defined on $H^1(\Omega)$. According to the well-known Sobolev
inequality and Sobolev trace inequality, we know that the functional
$I$ is well-defined and is of $C^2$ class. If both $q$ and $r$ are
subcritical, then the imbeddings $ H^1(\Omega)\hookrightarrow
L^r(\Omega)$ and $ H^1(\Omega)\hookrightarrow L^q(\partial \Omega)$
are compact. So it is easy to verify that $I$ satisfies the usual
$(PS)$ condition. That is, if $I(u_n)$ is bounded and $I'(u_n)\to
0$, then $u_n$ must have a convergent subsequence. However, if $r$
or(and) $q$ is critical, then the imbedding $
H^1(\Omega)\hookrightarrow L^r(\Omega)$ or(and) $
H^1(\Omega)\hookrightarrow L^q(\partial \Omega)$ is not compact. As
a result, the functional $I$ does not satisfies the $(PS)$
condition. To overcome this difficulty, one usually uses the
$(PS)_c$ condition to substitute the $(PS)$ condition. This method
has been widely used in dealing with elliptic equation involving
critical exponent, see \cite{BN}\cite{HYu}\cite{T} and etc. The
spirit of this paper is the same as the above works. We use $(PS)_c$
condition to substitute the usual $(PS)$ condition. However, due to
the different cases of $q,r$, we need to find different intervals so
that $I$ satisfies the $(PS)_c$ condition for $c$ in those
intervals. More precisely, if $r=\frac{2N}{N-2}$ and
$q<\frac{2(N-1)}{N-2}$, we will show that $I$ satisfies the $(PS)_c$
condition for $c\in (0,\frac{1}{2N}S^{\frac N2})$, where $S$ is the
usual Sobolev constant defined by
$$
S=\inf_{u\in D^{1,2}(\mathbb R^N),u\neq 0}\frac{\|\nabla
u\|^2_{L^2(\mathbb R^N)}}{\|u\|^2_{L^{2^*}(\mathbb R^N)}}.
$$
However, for the case $r<\frac{2N}{N-2}$ and $q=\frac{2(N-1)}{N-2}$,
we will show that $I$ satisfies the $(PS)_c$ condition for $c\in
(0,\frac{1}{2(N-1)}S_T^{N-1})$, where $S_T$ is the best constant of
Sobolev trace inequality which is defined by
$$
S_T=\inf_{u\in D^{1,2}(\mathbb R_+^N),u\neq 0}\frac{\|\nabla
u\|^2_{L^2(\R)}}{\|u(x',0)\|^2_{L^{2_*}(\partial\R)}}.
$$
The most difficult problem is the case in which both $r$ and $q$ are
critical, we will show that the functional $I$ satisfies the
$(PS)_c$ condition for $c\in (0,c_\infty)$, where $c_\infty$ is the
ground state level of equation \eqref{1.2}. In \cite{LZu}, the
authors proved that $c_\infty$ can be attained by some functions and
they gave all the expressions of these functions.

After the $(PS)_c$ condition is proved, then we need to show that
the Mountain Pass level of the functional $I$ indeed belongs to
these intervals. Then the standard Mountain Pass Theorem implies the
above theorems.

The rest of this paper is devoted to the proof of the above
theorems. In Section 2, we establish the multiple solutions result
for the subcritical problem, i.e., we prove Theorem \ref{t 1.1}. We
guess this result is well-known but we can't find the proper
reference so we give its proof to keep this paper self-contained. In
Section 3,4,5, we prove Theorem \ref{t 1.2}, Theorem \ref{t 1.3} and
Theorem \ref{t 1.4} respectively. In the following, we denote $C$ by
a positive constant, which may vary from line to line.
\section{Proof of Theorem \ref{t 1.1}}
In this section, we study the multiple solutions of problem
\eqref{1.1} under the subcritical assumptions, i.e.,
$2<r<\frac{2N}{N-2}$ and $2<q<\frac{2(N-1)}{N-2}$. We need the
following technical lemma, see Theorem 9.12 in \cite{Ra}.
\begin{lemma}\label{t 2.1}
Let $E$ be an infinite dimensional Banach space and let $I\in
C^2(E,\mathbb R)$ be even, satisfying $(PS)$, and $I(0)=0$. If
$E=V\oplus
X$, where $V$ is finite dimensional, and $I$ satisfies\\
(i) there exist constants $\rho,\alpha>0$, such that $I_{\partial
B_\rho\cap X}\geq \alpha$, and\\
(ii) for each finite dimensional subspace $\tilde E\subset E$, there
is an $R=R(\tilde E)$ such that $I\leq 0$ on $\tilde E\setminus
B_{R(\tilde E)}$,\\
then $I$ possesses an unbounded sequence of critical values.
\end{lemma}

To prove Theorem \ref{t 1.1}, we only need to verify that $I$
satisfies the conditions in Lemma \ref{t 2.1} under the assumptions
of Theorem \ref{t 1.1}. This is composed of the following lemmas. We
first show that $I$ satisfies the $(PS)$ condition.
\begin{lemma}\label{t 2.2}
Suppose that $2<r<\frac{2N}{N-2}$ and $2<q<\frac{2(N-1)}{N-2}$, then
$I$ satisfies the $(PS)$ condition.
\end{lemma}
\begin{proof}
Let $\{u_n\}$ be a (PS) sequence, that is $|I(u_n)|$ is bounded and
$I'(u_n)\to 0$, we need to show that $\{u_n\}$ has a convergent
subsequence.

We first show that $\{u_n\}$ is bounded. By means of $(PS)$
sequence, we have
\begin{equation}\label{2.1}
\frac 12\into |\nabla u_n|^2+u_n^2\,dx-\frac 1r\into
|u_n|^r\,dx-\frac 1{q}\intpo |u_n|^q\,dS\leq C
\end{equation}
and
\begin{equation}\label{2.2}
\into |\nabla u_n|^2+u_n^2\,dx-\into |u_n|^r\,dx-\intpo
|u_n|^q\,dS=o(1)\|u_n\|.
\end{equation}
If $r\geq q$, then we infer from the above two equations that
$$
(\frac 12-\frac 1q)\|u_n\|^2+(\frac 1q-\frac 1r) \into
|u_n|^r\,dx\leq C+o(1)\|u_n\|,
$$
which implies that $\|u_n\|$ is bounded. On the other hand, if
$r<q$, then we can still deduce from equation \eqref{2.1} and
equation \eqref{2.2} that
$$
(\frac 12-\frac 1r)\|u_n\|^2+(\frac 1r-\frac 1q) \intpo
|u_n|^q\,dS\leq C+o(1)\|u_n\|,
$$
which also implies that $\|u_n\|$ is bounded. So we proved that
$\{u_n\}$ is bounded.

Next, we show that $\{u_n\}$ has a convergent subsequence. Since
$\{u_n\}$ is bounded, we can suppose that, up to a subsequence,
$u_n\rightharpoonup u$. Then the Sobolev compact imbedding theorem
implies that
$$
\into |u_n|^r\,dx\to \into |u|^r\,dx
$$
and
$$
\intpo |u_n|^q\,dS\to \intpo |u|^q\,dS.
$$
Hence, we deduce from equation \eqref{2.2} that
$$
\|u_n\|^2\to \into |u|^r\,dx+\intpo |u|^q\,dS=\|u\|^2.
$$
Since $u_n\rightharpoonup u$ and $\|u_n\|\to \|u\|$, we get that
$u_n\to u$ strongly in $H^1(\Omega)$. This proves this lemma.

\end{proof}

Let $V=\emptyset$ and $E=X=H^1(\Omega)$ in Lemma \ref{t 2.1}, our
next lemma shows that (i) of Lemma \ref{t 2.1} holds.
\begin{lemma}\label{t 2.3}
There exist constants $\rho,\alpha>0$, such that $I_{\partial
B_\rho\cap X}\geq \alpha$.
\end{lemma}
\begin{proof}
By the Sobolev embedding theorem and Sobolev trace inequality, we
deduce that
$$
I(u)\geq \frac 12\|u\|^2-C\|u\|^r-C\|u\|^q.
$$
Hence, we can choose $\alpha,\rho>0$ small enough, such that
$$
I_{\partial B_\rho\cap X}\geq \alpha>0.
$$
\end{proof}

Finally, we verify condition (ii) in Lemma \ref{t 2.1}, we have the
following result.
\begin{lemma}\label{t 2.4}
Let $\tilde E\subset H^1(\Omega)$ be a finite dimensional subspace,
then there exists $R=R(\tilde E)$ such that $I\leq 0$ on $\tilde
E\setminus B_{R(\tilde E)}$.
\end{lemma}
\begin{proof}
Since $\tilde E$ is finite dimensional, then any norms on $\tilde E$
are equivalent. So we have
$$
I(u)\leq C_1\|u\|^2-C_2\|u\|^r-C_3\|u\|^q
$$
for any $u\in \tilde E$. Moreover, we deduce from
$2<r<\frac{2N}{N-2}$ and $2<q<\frac{2(N-1)}{N-2}$ that there exists
$R=R(\tilde E)>0$ such that
$$
I(u)<0
$$
on $\tilde E\setminus B_{R(\tilde E)}$.
\end{proof}
\begin{proof}[Proof of Theorem \ref{t 1.1}:] It is easy to see that
$I\in C^2(H^1(\Omega),\mathbb R)$ is even. Moreover, we infer from
Lemma \ref{t 2.2}, Lemma \ref{t 2.3} and Lemma \ref{t 2.4} that $I$
satisfies all the conditions in Lemma \ref{t 2.1}. So we conclude
that problem \eqref{1.1} possesses infinitely many solutions under
the assumptions of Theorem \ref{t 1.1}.

\end{proof}
\section{Proof of Theorem \ref{t 1.2}}
In this section, we study the existence result for problem
\eqref{1.1} under the assumption that $r$ is the critical exponent
and $q$ is subcritical. We first have the following local
compactness result.
\begin{lemma}\label{t 3.1}
Suppose that $r=\frac{2N}{N-2}$ and $2<q<\frac{2(N-1)}{N-2}$, then
$I$ satisfies the $(PS)_c$ condition for $c\in (0,\frac
1{2N}S^{\frac N2})$, where $S$ is the best constant of Sobolev
inequality defined by
$$
S=\inf_{u\in D^{1,2}(\mathbb R^N),u\neq 0}\frac{\|\nabla
u\|^2_{L^2(\mathbb R^N)}}{\|u\|^2_{L^{2^*}(\mathbb R^N)}}.
$$
\end{lemma}
\begin{proof}
Let $\{u_n\}\subset H^1(\Omega)$ be a $(PS)_c$ sequence for $I$ with
$c\in (0,\frac 1{2N}S^{\frac N2})$, that is,
$$
I(u_n)\to c
$$
and
$$
I'(u_n)\to 0,
$$
we need to show that $\{u_n\}$ has a convergent subsequence.

We first show that $\{u_n\}$ is bounded. In fact, we get from the
above two equations that
\begin{equation}\label{3.1}
\frac 12\into |\nabla u_n|^2+u_n^2\,dx-\frac 1{2^*}\into
|u_n|^{2^*}\,dx-\frac 1{q}\intpo |u_n|^q\,dS\to c
\end{equation}
and
\begin{equation}\label{3.2}
\into |\nabla u_n|^2+u_n^2\,dx-\into |u_n|^{2^*}\,dx-\intpo
|u_n|^q\,dS=o(1)\|u_n\|.
\end{equation}
Moreover, the above two equations imply that
$$
(\frac 12-\frac 1q)\|u_n\|^2+(\frac 1q-\frac 1{2^*}) \into
|u_n|^{2^*}\,dx=c+o(1)+o(1)\|u_n\|,
$$
which implies that $\|u_n\|$ is bounded.

Next, we show that $\{u_n\}$ has a convergent subsequence. Since
$\{u_n\}$ is bounded, then we can suppose that, up to a subsequence,
$u_n\rightharpoonup u$ in $H^1(\Omega)$, $u_n\rightharpoonup u$ in
$L^{2^*}(\Omega)$, $u_n\to u $ in $L^2(\Omega)$ and $u_n\to u $ in
$L^q(\partial \Omega)$. Now let $v_n=u_n-u$, then we deduce from
Brezis-Lieb theorem \cite{BL} that
\begin{equation}\label{3.3}
\frac 12\into |\nabla v_n|^2\,dx-\frac 1{2^*}\into
|v_n|^{2^*}\,dx=c-I(u)+o(1)
\end{equation}
and
\begin{equation}\label{3.4}
\into |\nabla v_n|^2\,dx-\into |v_n|^{2^*}\,dx=o(1).
\end{equation}
Hence, we get
\begin{equation}\label{3.5}
\into |\nabla v_n|^2\,dx=\into |v_n|^{2^*}\,dx=N(c-I(u))+o(1).
\end{equation}

Obviously, to prove that $u_n\to u$, it is equivalent to prove that
$v_n\to 0$ in $H^1(\Omega)$. We prove this conclusion by
contradiction. Suppose on the contrary, that is, $v_n\not\to 0$ in
$H^1(\Omega)$, then we will show that this will lead to a
contradiction.

Let $\ep$ be any fixed positive constant, then by Lemma 2.1 in
\cite{W}, there exists $\delta=\delta(\ep)>0$, such that if
$diam(supp\{ \vp\})\leq \delta$, then
$$
\into|\nabla \vp|^2\,dx\geq (2^{-\frac 2N}S-\ep)[\into
|\vp|^{2^*}\,dx]^{\frac{2}{2^*}},
$$
where $diam(D)$ is the diameter of the set $D$. Now let
$\{\vp_\alpha\}_{\alpha=1}^{n_0}$ be a unit partition of $\bar
\Omega$ with $diam(supp\{ \vp_\alpha\})\leq \delta$ for each
$\alpha$. Since $\partial \Omega\in C^1$, then we deduce from Lemma
2.1 in \cite{W} that
$$
\into|\nabla (u\vp_\alpha)|^2\,dx\geq (2^{-\frac 2N}S-\ep)[\into
|u\vp_\alpha|^{2^*}\,dx]^{\frac{2}{2^*}}.
$$
Hence, we have
\begin{eqnarray*}
[N(c-I(u)+o(1))]^{\frac 2{2^*}}&=& (\into (v_n)^{2^*}\,dx)^{\frac 2{2^*}}\\
&=&\|v_n^2\|_{L^{\frac{2^*}2}(\Omega)} \\
   &=& \|\sum_{\alpha=1}^{n_0}\vp_\alpha
   v_n^2\|_{L^{\frac{2^*}2}(\Omega)}\\
   &\leq & \sum_{\alpha=1}^{n_0}\|\vp_\alpha
   v_n^2\|_{L^{\frac{2^*}2}(\Omega)}\\
   &\leq &(2^{-\frac 2N}S-\ep)^{-1}\sum_{\alpha=1}^{n_0}\into |\nabla(v_n\vp_\alpha^{\frac
   12})|^2\,dx\\&\leq& (2^{-\frac 2N}S-\ep)^{-1}[(1+\ep)\into |\nabla v_n|^2\,dx+C\into
   v_n^2\,dx]\\&=&(2^{-\frac 2N}S-\ep)^{-1}(1+\ep)\into |\nabla
   v_n|^2\,dx+o(1)\\&=&(2^{-\frac
   2N}S-\ep)^{-1}(1+\ep)[N(c-I(u)+o(1))].
\end{eqnarray*}
Moreover, we infer from the above inequality that
\begin{equation}\label{3.6}
[c-I(u)+o(1)]\geq \frac 1N[\frac{2^{-\frac 2N}S-\ep}{1+\ep}]^{\frac
N2}.
\end{equation}
We note that $I(u)\geq 0$ since $u$ is a solution of problem
\eqref{1.1}. So if we let $\ep\to 0$ and $n\to \infty$ in equation
\eqref{3.6}, then we get
$$
[c-I(u)]\geq \frac 1{2N}S^{\frac N2},
$$
which contradicts that $c\in (0, \frac 1{2N}S^{\frac N2})$. This
finishes the proof of this lemma.
\end{proof}

We want to use the Mountain Pass theorem in \cite{Ra} to prove the
existence result. By the above local $(PS)$ condition, we need to
prove that the Mountain Pass level of $I$ is indeed below $ \frac
1{2N}S^{\frac N2}$. Since $\Omega$ is bounded, then there exists a
ball $B_R(\bar x)$ containing $\Omega$ and $\partial B_R(\bar x)\cap
\bar \Omega\neq \emptyset$. Suppose that $x_0\in \partial B_R(\bar
x)\cap \bar \Omega$, then we have $2\alpha_i\geq \frac 1R$ for each
$1\leq i\leq N-1$, where $2\alpha_i(i=1,\cdots,N-1)$ are the
principal curvatures of $\partial \Omega$ at $x_0$. We can suppose
$x_0=0$ and $\Omega\subset \{x:x_N>0\}$ without loss of generality,
then the boundary of $ \Omega$ can be represented by
\begin{equation}\label{3.7}
x_N=h(x')= \sum_{i=1}^{N-1}\alpha_i x_i^2+o(|x'|^2),\quad \forall\
x'=(x_1,x_2,\cdots,x_{N-1})\in D(0,\delta)
\end{equation}
for some $\delta>0$, where $ D(0,\delta)=B_\delta
(0)\cap\{x:x_N=0\}$. Set
\begin{equation}\label{3.8}
u_\ep(x)=\frac{[N(N-2)\ep^2]^{\frac{N-2}{4}}}{[\ep^2+|x|^2]^{\frac{N-2}2}}\quad
{\rm in}\ \R,
\end{equation}
then it is well-known that $u_\ep$ solves equation
\begin{equation}\label{3.9}
 \left\{
\begin{array}{ll}
-\Delta u=|u|^{2^*-2}u &\text{in} \; \R,\\
\\ \frac{\partial u}{\partial \nu}=0 &\text{on}\;\partial\R.
\end{array}
\right.
\end{equation}
Moreover, we have $ \intr |\nabla u_\ep|^2\,dx= \intr
u_\ep^{\frac{2N}{N-2}}\,dx=\frac 1{2}S^{\frac N2}$, where $S$ is the
best constant of the Sobolev inequality.

With the above notations, we have the following estimates.
\begin{lemma}\label{t 3.2}
Let $\Omega\subset \mathbb R^N$ with $N\geq 4$ as above, $u_\ep$ be
defined by equation
\eqref{3.8} and $g(x')=\sum_{i=1}^{N-1}\alpha_i x_i^2$, then we have\\
(i) $\into |\nabla u_\ep|^2\,dx=\intr |\nabla
u_\ep|^2\,dx-N^{\frac{N-2}2}(N-2)^{\frac{N+2}2}\ep\intrr
\frac{|x'|^2g(x')}{[1+|x'|^2]^N}\,dx'+o(\ep)$;\\
(ii) $\into u_\ep^{\frac{2N}{N-2}}\,dx=\intr
u_\ep^{\frac{2N}{N-2}}-N^{\frac N2}(N-2)^{\frac N2}\ep\intrr
\frac{g(x')}{[1+|x'|^2]^N}\,dx'+o(\ep)$;\\
(iii) $\intpo u_\ep^q=\intrr
u_\ep(x',0)^{q}\,dx'+o(\ep)=C\ep^{(N-1)-\frac{N-2}2q}+o(\ep^{(N-1)-\frac{N-2}2q})$
for $q\in
(2,2_*)$;\\
(iv) $ \into u_\ep^2\,dx=\left\{
\begin{array}{ll}
O(\ep^2|\ln \ep|) &\text{if}\ \; N=4,\\
\\ O(\ep^2) &\text{if}\ \; N\geq 5.
\end{array}
\right.$
 \end{lemma}
\begin{proof}
A direct calculation implies that
\begin{equation}\label{3.10}
\begin{split}
&\intr |\nabla u_\ep|^2\,dx-\into |\nabla u_\ep|^2\,dx\\&=\intrr
dx'\int_0^{g(x')}N^{\frac{N-2}{2}}(N-2)^{\frac{N+2}2}\ep^{N-2}\frac{|x|^2}{[\ep^2+|x|^2]^N}\,dx_N\\&+
\int_{D(0,\delta)}
dx'\int_{g(x')}^{h(x')}N^{\frac{N-2}{2}}(N-2)^{\frac{N+2}2}\ep^{N-2}\frac{|x|^2}{[\ep^2+|x|^2]^N}\,dx_N+O(\ep^{N-2})
\\&=\intrr
dx'\int_0^{\ep
g(x')}N^{\frac{N-2}{2}}(N-2)^{\frac{N+2}2}\frac{|x|^2}{[1+|x|^2]^N}\,dx_N\\&+
\int_{D(0,\delta)}
dx'\int_{g(x')}^{h(x')}N^{\frac{N-2}{2}}(N-2)^{\frac{N+2}2}\ep^{N-2}\frac{|x|^2}{[\ep^2+|x|^2]^N}\,dx_N+O(\ep^{N-2})\\
&=N^{\frac{N-2}{2}}(N-2)^{\frac{N+2}2}\ep \intrr
\frac{|x'|^2g(x')}{[1+|x'|^2]^N}\,dx'+o(\ep)\\&+ \int_{D(0,\delta)}
dx'\int_{g(x')}^{h(x')}N^{\frac{N-2}{2}}(N-2)^{\frac{N+2}2}\ep^{N-2}\frac{|x|^2}{[\ep^2+|x|^2]^N}\,dx_N.
\end{split}
\end{equation}
 On the other hand, since $h(x')=g(x')+o(|x'|^2)$ in $D(0,\delta)$,
 then
it follows that $\forall \sigma>0$, there exists $C(\sigma)>0$ such
that $|h(x')-g(x')|\leq \sigma|x'|^2+C(\sigma)|x'|^{\frac 52}$ for
$x'\in D(0,\delta)$. Therefore, we have
\begin{eqnarray*}
&&\int_{D(0,\delta)}
dx'\int_{g(x')}^{h(x')}N^{\frac{N-2}{2}}(N-2)^{\frac{N+2}2}\ep^{N-2}\frac{|x|^2}{[\ep^2+|x|^2]^N}\,dx_N\\&\leq&
C\ep^{N-2}\int_{D(0,\delta)}\,dx'\int_{g(x')}^{h(x')}\frac{|x|^2}{[\ep^2+|x|^2]^N}\,dx_N\\&\leq&
C\ep^{N-2}\int_{D(0,\delta)}\frac{|h(x')-g(x')|}{[\ep^2+|x'|^2]^{N-1}}\,dx'\\&\leq&
C\ep^{N-2}\int_{D(0,\delta)}\frac{\sigma |x'|^2+C(\sigma)|x'|^{\frac
52}}{[\ep^2+|x'|^2]^{N-1}}dx'\\&\leq
&C\ep(\sigma+C(\sigma)\ep^{\frac 12})
\end{eqnarray*}
since $N\geq 4$. Moreover, since $\sigma$ is arbitrary, we have
\begin{equation}\label{3.11}
\int_{D(0,\delta)}
dx'\int_{g(x')}^{h(x')}N^{\frac{N-2}{2}}(N-2)^{\frac{N+2}2}\ep^{N-2}\frac{|x|^2}{[\ep^2+|x|^2]^N}\,dx_N=o(\ep).
\end{equation}
Finally, we deduce from equation \eqref{3.10} and equation
\eqref{3.11} that
$$
\into |\nabla u_\ep|^2\,dx=\intr |\nabla
u_\ep|^2\,dx-N^{\frac{N-2}2}(N-2)^{\frac{N+2}2}\ep\intrr
\frac{|x'|^2g(x')}{[1+|x'|^2]^N}\,dx'+o(\ep),
$$
which proves (i).

For (ii), we have
\begin{equation}\label{3.12}
\begin{split}
&\intr u_\ep^{2^*}\,dx-\into u_\ep^{2^*}\,dx\\&=\intrr
dx'\int_0^{g(x')}N^{\frac N2}(N-2)^{\frac
N2}\ep^N[\ep^2+|x|^2]^{-N}\,dx_N\\&+\int_{D(0,\delta)}
dx'\int_{g(x')}^{h(x')}N^{\frac N2}(N-2)^{\frac
N2}\ep^N[\ep^2+|x|^2]^{-N}\,dx_N+O(\ep^{N}).
\end{split}
\end{equation}
A direct calculation implies
\begin{equation}\label{3.13}
\begin{split}
&\int_{D(0,\delta)} dx'\int_0^{g(x')}N^{\frac N2}(N-2)^{\frac
N2}\ep^N[\ep^2+|x|^2]^{-N}\,dx_N\\&= \intrr
dx'\int_0^{g(x')}N^{\frac N2}(N-2)^{\frac
N2}\ep^N[\ep^2+|x|^2]^{-N}\,dx_N+O(\ep^N)
\\&=N^{\frac N2}(N-2)^{\frac
N2}\intrr \,dx'\int_0^{\ep g(x')}\frac
1{[1+|x|^2]^N}\,dx_N+o(\ep)\\&=N^{\frac N2}(N-2)^{\frac N2}\ep\intrr
\frac{g(x')}{[1+|x'|^2]^N}\,dx'+o(\ep).
\end{split}
\end{equation}
Moreover, similar to equation \eqref{3.11}, we have
\begin{equation}\label{3.14}
\int_{D(0,\delta)} dx'\int_{g(x')}^{h(x')}N^{\frac N2}(N-2)^{\frac
N2}\ep^N[\ep^2+|x|^2]^{-N}\,dx_N=o(\ep).
\end{equation}
Hence, it follows from equations \eqref{3.12},\eqref{3.13} and
\eqref{3.14} that
$$
\into u_\ep^{\frac{2N}{N-2}}\,dx=\intr
u_\ep^{\frac{2N}{N-2}}\,dx-N^{\frac N2}(N-2)^{\frac N2}\ep\intrr
\frac{g(x')}{[1+|x'|^2]^N}\,dx'+o(\ep),
$$
which proves (ii).

As for (iii), we note that since $q\in (2,2_*)$, so we get that
\begin{equation}\label{3.15}
\begin{split}
\intrr u_\ep(x',0)^q\,dx'&=C\ep^{\frac{N-2}2q}\intrr \frac
1{[\ep^2+|x'|^2]^{\frac
{N-2}2q}}\,dx'\\&=C\ep^{N-1-\frac{N-2}2q}\intrr \frac
1{[1+|x'|^2]^{\frac {N-2}2q}}\,dx'.
\end{split}
\end{equation}
Moreover, since $q\in (2,\frac{2(N-1)}{N-2})$, we have
$N-1-\frac{N-2}2q\in (0,1)$, hence we get
\begin{equation}\label{3.16}
\intpo u_\ep^q\,dS=
C\ep^{N-1-\frac{N-2}2q}+o(\ep^{(N-1)-\frac{N-2}2q}),
\end{equation}
which proves (iii).

(iv) is proved in \cite{BN}.

\end{proof}
With the above preparations, we can prove Theorem \ref{t 1.2} with
$N\geq 4$ now.
\begin{proof}[Proof of Theorem \ref{t 1.2} with $N\geq 4$:]
We use the Mountain Pass theorem to prove our result. First, we
infer from Sobolev imbedding theorem, Sobolev trace inequality that
$$
I(u)\geq \frac 12\|u\|^2-C\|u\|^{2^*}-C\|u\|^{q},
$$
hence, there exist $\rho,\alpha>0$ such that
$$
I(u)\geq \alpha>0
$$
for $u\in \partial B_\rho$. Moreover, let $u_0\neq 0$ fixed, then it
is easy to check that
$$
I(tu_0)\to -\infty
$$
as $t\to \infty$. So there exists $t_0>0$ such that
$\|t_0u_0\|>\rho$ and
$$
I(t_0u_0)<0.
$$
Define
$$
c=\inf_{\gamma\in \Gamma}\sup_{t\in [0,1]}I(\gamma(t))
$$
with
$$
\Gamma=\{\gamma\in C([0,1], H^1(\Omega)): \gamma(0)=0,
\gamma(1)=t_0u_0\},
$$
then $c$ is a well-defined positive constant. In order to show that
$c$ is indeed a critical value for $I$, we only need to show that
$$
c<\frac 1{2N}S^{\frac N2}.
$$
For this purpose, we note that
$$
I(t u_\ep)=\frac {t^2}2\|u_\ep\|^2-\frac{t^{2^*}}{2^*}\into
u_\ep^{2^*}\,dx-\frac {t^q}{q}\intpo u_\ep^q\,dS.
$$
Hence there exists $t_\ep>0 $ such that $I(t_\ep u_\ep)$ attains its
maximum at $t_\ep$. Moreover, it follows from Lemma \ref{t 3.2} that
$t_\ep\to 1$ as $\ep\to 0$. Hence, for $N\geq 5$, we get
\begin{eqnarray*}
I(t_\ep u_\ep)&\leq &\frac {t_\ep^2}2\intr |\nabla
u_\ep|^2\,dx-\frac{t_\ep^{2^*}}{2^*}\intr
|u_\ep|^{2^*}\,dx-\frac{t_\ep^2}{2}
N^{\frac{N-2}2}(N-2)^{\frac{N+2}2}\ep\intrr
\frac{|x'|^2g(x')}{[1+|x'|^2]^N}\,dx'\\&+&O(\ep^2)-C\ep^{N-1-\frac{N-2}2q}+C\ep\intrr
\frac{g(x')}{[1+|x'|^2]^N}\,dx'.
\end{eqnarray*}
On the other hand, since the function
$$
f(t)=\frac {t^2}2\intr |\nabla u_\ep|^2\,dx-\frac{t^{2^*}}{2^*}\intr
|u_\ep|^{2^*}\,dx
$$
attains its maximum at $t=1$ and the maximum is $\frac 1{2N}S^{\frac
N2}$, hence we have
\begin{equation}\label{3.17}
\begin{split}
I(t_\ep u_\ep)&\leq \frac 1{2N}S^{\frac N2}-C\ep\intrr
\frac{|x'|^2g(x')}{[1+|x'|^2]^N}\,dx'\\&+O(\ep^2)-C\ep^{N-1-\frac{N-2}2q}+C\ep\intrr
\frac{g(x')}{[1+|x'|^2]^N}\,dx'.
\end{split}
\end{equation}
Since $q\in (2,\frac{2(N-1)}{N-2})$, we get $N-1-\frac{N-2}2q\in
(0,1)$. Finally, we infer from the above inequality that
$$
I(t_\ep u_\ep)< \frac 1{2N}S^{\frac N2}
$$
for $\ep$ small enough.

The case $N=4$ is similar, we only need to replace $O(\ep^2)$ by
$O(\ep^2|\ln \ep|)$ in equation \eqref{3.17}. The rest of the proof
is the same as $N\geq 5$.

\end{proof}

Finally, we study the case $N=3$. In this case, we suppose the
principal curvatures of $\partial \Omega$ at $x_0 \in\partial
\Omega$ belong to interval $(2a,2A)$ for some $0<a\leq A<\infty$.
Then we have $a|x'|^2\leq h(x')\leq A |x'|^2$ for $x'\in
D(0,\delta)$. With these notations, we have the following estimates.
\begin{lemma}\label{t 3.3}
Suppose $N=3$ and $x_0$ as above, then we have the following estimates.\\
(i) $\into |\nabla u_\ep|^2\,dx\leq \intr |\nabla
u_\ep|^2\,dx-C\ep|\ln
\ep|+O(\ep)$;\\
(ii) $\into u_\ep^6\,dx\geq \intr u_\ep^6\,dx-C\ep.$\\
(iii) $\intpo u_\ep^q\,dS\geq C\ep^{2-\frac q2}$ for $q\in (2,4)$.\\
(iv) $\into u_\ep^2\,dx'=O(\ep)$.
\end{lemma}
\begin{proof}
As for (i), we infer from the above setting that
\begin{equation}\label{3.18}
\begin{split}
&\into |\nabla u_\ep|^2\,dx=\intr |\nabla
u_\ep|^2\,dx-\int_{D(0,\delta)}dx'\int_0^{h(x')}|\nabla
u_\ep|^2\,dx_N+O(\ep)\\ &\leq \intr |\nabla
u_\ep|^2\,dx-\int_{D(0,\delta)}dx'\int_0^{a|x'|^2}|\nabla
u_\ep|^2\,dx_N+O(\ep)\\&\leq \intr |\nabla u_\ep|^2\,dx-
C\ep\int_{D(0,\delta)}\frac{a|x'|^4}{(\ep^2+|x'|^2)^N}\,dx'+O(\ep)\\&\leq
\intr |\nabla u_\ep|^2\,dx-C\ep|\ln \ep|+O(\ep).
\end{split}
\end{equation}
As for (ii), we have
\begin{equation}\label{3.19}
\begin{split}
\into u_\ep^6\,dx&=\intr
u_\ep^6\,dx-\int_{D(0,\delta)}dx'\int_0^{h(x')}u_\ep^6\,dx_N+O(\ep^{3})\\&\geq
\intr
u_\ep^6\,dx-\int_{D(0,\delta)}dx'\int_0^{A|x'|^2}u_\ep^6\,dx_N+O(\ep^{3})\\&
\geq \intr u_\ep^6\,dx-C\ep\intrr
\frac{|x'|^2}{[1+|x'|^2]^3}\,dx'+O(\ep^{3})\\&\geq \intr
u_\ep^6\,dx-C\ep.
\end{split}
\end{equation}

(iii) is the same as Lemma \ref{t 3.2} and (iv) can be found in
\cite{BN}. This finishes the proof of this lemma.

\end{proof}
\begin{proof}[Proof of Theorem \ref{t 1.2} with $N=3$:]
By the same reason as for $N\geq 4$, we conclude that the functional
$I$ possesses the Mountain Pass structure, so we only need to show
that the Mountain Pass level is below $\frac 1{2N}S^{\frac N 2}$.
For this purpose, we let $u_\ep, t_\ep$ the same as the case $N\geq
4$, then we have
\begin{equation}\label{3.20}
\begin{split}
I(t_\ep u_\ep)&\leq \frac 1{2N}S^{\frac N2}-C\ep|\ln
\ep|+C\ep-C\ep^{2-\frac q2}.
\end{split}
\end{equation}
Since $q\in (2,4)$, we deduce that $2-\frac q2\in (0,1)$. Insert
this into the above inequality, then we conclude that
$$
I(t_\ep u_\ep)< \frac 1{2N}S^{\frac N2}
$$
for $\ep>0$ small enough. This finishes the proof of Theorem \ref{t
1.2} for $ N=3$.
\end{proof}

\section{Proof of Theorem \ref{1.3}}
In this section, we study the existence result of another critical
equation, i.e., the case $2<r<\frac{2N}{N-2}$ and
$q=\frac{2(N-1)}{N-2}$. We first introduce the respective "limit
equation" corresponding to equation \eqref{1.1} in this case.
Consider the following equation
\begin{equation}\label{4.1}
 \left\{
\begin{array}{ll}
-\Delta u=0&\text{in} \; \R,\\
\\ \frac{\partial u}{\partial \nu}=|u|^{2_*-2}u  &\text{on}\;\partial\R,
\end{array}
\right.
\end{equation}
its ground state solutions are closely related to the following
Sobolev trace inequality
\begin{equation}\label{4.2}
S_T=\inf_{u\in D^{1,2}(\R):u\neq 0}\frac{\intr |\nabla
u|^2\,dx}{(\intrr |u(x',0)|^{2_*}\,dx')^{\frac 2{2_*}}}.
\end{equation}
That is, the minimizers of $S_T$ multiplied by a constant are the
solutions of equation \eqref{4.1}. Moreover, the ground state
solutions of equation \eqref{4.1} have the following form
\begin{equation}\label{4.3}
u_\ep(x)=\frac{(N-2)^{\frac{N-2}2}\ep^{\frac
{N-2}2}}{[(\ep+x_N)^2+|x'|^2]^{\frac{N-2}2}}.
\end{equation}

With the above preparations, we can study the existence result of
equation \eqref{1.1} in this case now. Obviously, the solutions of
problem \eqref{1.1} in this situation correspond to the critical
points of the functional
\begin{equation}\label{4.4}
I(u)=\frac 12\into |\nabla u|^2+u^2\,dx-\frac 1r\into
|u|^r\,dx-\frac 1{2_*}\intpo |u|^{2_*}\,dS
\end{equation}
defined on $H^1(\Omega)$. First, we have the following compactness
lemma.
\begin{lemma}\label{t 4.1}
Under the assumptions of Theorem \ref{t 1.3}, the functional $I$
satisfies the $(PS)_c$ condition for $c\in (0,\frac
1{2(N-1)}S_T^{N-1})$.
\end{lemma}
\begin{proof}
Let $\{u_n\}\subset H^1(\Omega)$ be a $(PS)_c$ sequence for $I$ with
$c\in (0,\frac 1{2(N-1)}S_T^{N-1})$, that is,
$$
I(u_n)\to c
$$
and
$$
I'(u_n)\to 0,
$$
we will prove that $\{u_n\}$ has a convergent subsequence.

We first show that $\{u_n\}$ is bounded. For this purpose, we infer
from the above two equalities that
\begin{equation}\label{4.5}
\frac 12\into |\nabla u_n|^2+u_n^2\,dx-\frac 1{r}\into
|u_n|^{r}\,dx-\frac 1{2_*}\intpo |u_n|^{2_*}\,dS\to c
\end{equation}
and
\begin{equation}\label{4.6}
\into |\nabla u_n|^2+u_n^2\,dx-\into |u_n|^{r}\,dx-\intpo
|u_n|^{2_*}\,dS=o(1)\|u_n\|.
\end{equation}
If $r\leq \frac{2(N-1)}{N-2}$, then we deduce from the above two
equations that
$$
(\frac 12-\frac 1r)\|u_n\|^2+(\frac 1r-\frac 1{2_*}) \intpo
|u_n|^{2_*}\,dS=c+o(1)+o(1)\|u_n\|,
$$
which implies that $\|u_n\|$ is bounded. Similarly, if $r>2_*$, then
we can still get from equation \eqref{4.5} and equation \eqref{4.6}
that
$$
(\frac 12-\frac 1{2_*})\|u_n\|^2+(\frac 1{2_*}-\frac 1r) \into
|u_n|^{r}\,dx=c+o(1)+o(1)\|u_n\|,
$$
which also implies that $\|u_n\|$ is bounded. So in both cases, we
conclude that $\|u_n\|$ is bounded.

Next, we show that $\{u_n\}$ has a convergent subsequence. Since
$\{u_n\}$ is bounded, we can suppose that, up to a subsequence,
$u_n\rightharpoonup u$ in $H^1(\Omega)$, $u_n\to u$ in
$L^{r}(\Omega)$, $u_n\to u $ in $L^2(\Omega)$ and
$u_n\rightharpoonup u $ in $L^{2_*}(\partial \Omega)$. Now let
$v_n=u_n-u$, then we deduce from Brezis-Lieb theorem \cite{BL} that
\begin{equation}\label{4.7}
\frac 12\into |\nabla v_n|^2\,dx-\frac 1{2_*}\intpo
|v_n|^{2_*}\,dS=c-I(u)+o(1)
\end{equation}
and
\begin{equation}\label{4.8}
\into |\nabla v_n|^2\,dx-\intpo |v_n|^{2_*}\,dS=o(1).
\end{equation}
Hence, we get
\begin{equation}\label{4.9}
\into |\nabla v_n|^2\,dx=\intpo |v_n|^{2_*}\,dS=2(N-1)(c-I(u))+o(1).
\end{equation}
As before, to prove this lemma, it is equivalent to prove $v_n\to 0$
in $H^1(\Omega)$. We prove this conclusion by contradiction. Suppose
on the contrary, that is, $v_n\not\to 0$ in $H^1(\Omega)$, then we
will show that this will lead to a contradiction. In fact, we have
\begin{eqnarray*}
[2(N-1)(c-I(u)+o(1))]^{\frac 2{2_*}}&=& (\intpo (v_n)^{2_*}\,dS)^{\frac 2{2_*}}\\
&\leq & \frac 1{S_T}\into |\nabla v_n|^2\,dx\\&=&\frac
1{S_T}[2(N-1)(c-I(u)+o(1))].
\end{eqnarray*}
It follows from the above inequality that
$$
2(N-1)(c-I(u)+o(1))\geq S_T^{N-1}.
$$
Let $n\to \infty$ in the above inequality and note that $I(u)\geq
0$, then we get that
$$
c\geq \frac 1{2(N-1)}S^{N-1},
$$
which contradicts the choice of $c$. This finishes the proof of this
lemma.
\end{proof}

In order to prove the Mountain Pass level of functional $I$ is
indeed below $\frac 1{2(N-1)}S^{N-1}$, we need to estimate the
Mountain Pass value carefully. As before, we assume that the ball
$B_R(\bar x)$ contains $\Omega$ and $\partial B_R(\bar x)\cap \bar
\Omega\neq \emptyset$. Suppose that $x_0\in
\partial B_R(\bar x)\cap \bar \Omega$, then we have $2\alpha_i\geq
\frac 1R$ for each $1\leq i\leq N-1$, where
$2\alpha_i(i=1,\cdots,N-1)$ are the principal curvatures of
$\partial \Omega$ at $x_0$. We can suppose $x_0=0$ and
$\Omega\subset \{x:x_N>0\}$ without loss of generality, then the
boundary of $ \Omega$ can be represented by
\begin{equation}\label{4.10}
x_N=h(x')= \sum_{i=1}^{N-1}\alpha_i x_i^2+o(|x'|^2),\quad \forall\
x'=(x_1,x_2,\cdots,x_{N-1})\in D(0,\delta)
\end{equation}
for some $\delta>0$, where we denote
$D(0,\delta)=B_\delta(0)\cap\{x_N=0\}$. In the following, we denote
$g(x')=\sum_{i=1}^{N-1}\alpha_i x_i^2$ for simplicity.
\begin{lemma}\label{t 4.2}
Let $N\geq 4$ and $u_\ep$ be defined in equation \eqref{4.3}, then
we have the
following estimates.\\
(i) $\into |\nabla u_\ep|^2\,dx=\intr |\nabla
u_\ep|^2\,dx-(N-2)^N\ep\intrr
\frac{g(x')}{[1+|x'|^2]^{N-1}}\,dx'+o(\ep)$;\\
(ii) $\intpo u_\ep^{2_*}\,dS=\intrr
u_\ep(x',0)^{2_*}\,dx'-2(N-1)(N-2)^{N-1}\ep\intrr
\frac{g(x')}{[1+|x'|^2]^N}\,dx'+o(\ep)$;\\
(iii) $ \into u_\ep^2\,dx=\left\{
\begin{array}{ll}
O(\ep^2|\ln \ep|) &\text{if}\ \; N=4,\\
\\ O(\ep^2) &\text{if}\ \; N\geq5.
\end{array}
\right.$
\end{lemma}
\begin{proof}
For (i), a direct calculation implies that
\begin{equation}\label{4.11}
\begin{split}
&\intr |\nabla u_\ep|^2\,dx-\into |\nabla u_\ep|^2\,dx\\&=\intrr
dx'\int_0^{g(x')}(N-2)^{N}\ep^{N-2}\frac{1}{[(\ep+x_N)^2+|x'|^2]^{N-1}}\,dx_N\\&+
\int_{D(0,\delta)}
dx'\int_{g(x')}^{h(x')}(N-2)^{N}\ep^{N-2}\frac{1}{[(\ep+x_N)^2+|x'|^2]^{N-1}}\,dx_N+O(\ep^{N-2})
\\&=\intrr
dx'\int_0^{\ep
g(x')}(N-2)^{N}\frac{1}{[(1+x_N)^2+|x'|^2]^{N-1}}\,dx_N\\&+
\int_{D(0,\delta)}
dx'\int_{g(x')}^{h(x')}(N-2)^{N}\ep^{N-2}\frac{1}{[(\ep+x_N)^2+|x'|^2]^{N-1}}\,dx_N+O(\ep^{N-2})\\
&=(N-2)^{N}\ep \intrr \frac{g(x')}{[1+|x'|^2]^{N-1}}\,dx'\\&+
\int_{D(0,\delta)}
dx'\int_{g(x')}^{h(x')}(N-2)^{N}\ep^{N-2}\frac{1}{[(\ep+x_N)^2+|x'|^2]^{N-1}}\,dx_N+o(\ep).
\end{split}
\end{equation}
 On the other hand, since $h(x')=g(x')+o(|x'|^2)$,
it follows that $\forall \sigma>0$, there exists $C(\sigma)>0$ such
that $|h(x')-g(x')|\leq \sigma|x'|^2+C(\sigma)|x'|^{\frac 52}$ for
$x'\in D(0,\delta)$. Therefore, we have
\begin{eqnarray*}
&&\int_{D(0,\delta)}
dx'\int_{g(x')}^{h(x')}(N-2)^{N}\ep^{N-2}\frac{1}{[(\ep+x_N)^2+|x|^2]^{N-1}}\,dx_N\\&\leq&
C\ep^{N-2}\int_{D(0,\delta)}\,dx'\int_{g(x')}^{h(x')}\frac{1}{[(\ep+x_N)^2+|x|^2]^{N-1}}\,dx_N\\&\leq&
C\ep^{N-2}\int_{D(0,\delta)}\frac{|h(x')-g(x')|}{[\ep^2+|x'|^2]^{N-1}}\,dx'\\&\leq&
C\ep^{N-2}\int_{D(0,\delta)}\frac{\sigma |x'|^2+C(\sigma)|x'|^{\frac
52}}{[\ep^2+|x'|^2]^{N-1}}\,dx'\\&\leq
&C\ep(\sigma+C(\sigma)\ep^{\frac 12}).
\end{eqnarray*}
Since $\sigma$ is arbitrary, we have
\begin{equation}\label{4.12}
\int_{D(0,\delta)}
dx'\int_{g(x')}^{h(x')}(N-2)^{N}\ep^{N-2}\frac{1}{[(\ep+x_N)^2+|x'|^2]^{N-1}}\,dx_N=o(\ep).
\end{equation}
Finally, we deduce from equation \eqref{4.11} and equation
\eqref{4.12} that
$$
\into |\nabla u_\ep|^2\,dx=\intr |\nabla
u_\ep|^2\,dx-(N-2)^{N}\ep\intrr
\frac{g(x')}{[1+|x'|^2]^{N-1}}\,dx'+o(\ep),
$$
this proves (i).

As for (ii), if we denote
$$
f(r)=u_\ep(x',rg(x'))^{\frac{2(N-1)}{N-2}}\sqrt{1+r^2|\nabla
g(x')|^2},
$$
then we have
\begin{equation}\label{4.14}
\begin{split}
&\intpo u_\ep^{2_*}\,dS-\intrr u_\ep(x',0)^{2_*}\,dx'\\&=\intrr
[f(1)-f(0)]\,dx'+o(\ep)\\&=\intrr f'(r_\ep)dx'+o(\ep)\\&=-\intrr
2(N-1)(N-2)^{N-1}\ep^{N-1}\frac{(\ep+r_\ep g(x'))g(x')}{[(\ep+r_\ep
g(x'))^2+|x'|^2]^N}\sqrt{1+r_\ep^2|\nabla
g(x')|^2}\,dx'+o(\ep)\\&=-2(N-1)(N-2)^{N-1}\ep\intrr
\frac{g(x')}{[1+|x'|^2]^N}\,dx'+o(\ep),
\end{split}
\end{equation}
which proves (iii)

(iii) can be proved as in \cite{BN}.

\end{proof}

With the above preparations, we can prove Theorem \ref{t 1.3} for
$N\geq 4$ now.
\begin{proof}[Proof of Theorem \ref{t 1.3} with $N\geq 4$:]
We use the Mountain Pass theorem to prove our result. First, we note
that
$$
I(u)\geq \frac 12\|u\|^2-C\|u\|^{r}-C\|u\|^{2_*},
$$
hence, there exist $\rho,\alpha>0$ such that
$$
I(u)\geq \alpha>0
$$
for $u\in \partial B_\rho$. Moreover, let $u_0\in H^1(\Omega)$ and
$u_0\neq 0$ fixed, then it is easy to check that
$$
I(tu_0)\to -\infty
$$
as $t\to \infty$. So there exists $t_0>0$ such that
$\|t_0u_0\|>\rho$ and
$$
I(t_0u_0)<0.
$$
Define
$$
c=\inf_{\gamma\in \Gamma}\sup_{t\in [0,1]}I(\gamma(t))
$$
with
$$
\Gamma=\{\gamma\in C([0,1], H^1(\Omega)): \gamma(0)=0,
\gamma(1)=t_0u_0\},
$$
then $c$ is a well-defined positive constant. In order to show that
$c$ is indeed a critical value for $I$, we only need to show that
$$
c<\frac 1{2(N-1)}S_T^{N-1}.
$$
For this purpose, we note that
$$
I(t u_\ep)=\frac {t^2}2\|u_\ep\|^2-\frac{t^{r}}{r}\into
u_\ep^{r}\,dx-\frac {t^{2_*}}{{2_*}}\intpo u_\ep^{2_*}\,dS.
$$
Hence there exists $t_\ep $ such that $I(t_\ep u_\ep)$ attains its
maximum at $t_\ep$. Moreover, it follows from Lemma \ref{t 4.2} that
$t_\ep\to 1$ as $\ep\to 0$. Hence, for $N\geq 5$, we get
\begin{equation}\label{4.15}
\begin{split}
I(t_\ep u_\ep)&\leq \frac {t_\ep^2}2\intr |\nabla
u_\ep|^2\,dx-\frac{t_\ep^{2_*}}{2_*}\intrr
|u_\ep(x',0)|^{2_*}\,dx'-\frac{t_\ep^2}{2} (N-2)^{N}\ep\intrr
\frac{g(x')}{[1+|x'|^2]^{N-1}}\,dx'\\&-\frac {t_\ep^r}{r}\into
u_\ep^r\,dx+t_\ep^{2_*}(N-2)^N\ep\intrr
\frac{g(x')}{[1+|x'|^2]^{N}}\,dx'+O(\ep^2).
\end{split}
\end{equation}
We note that the function
$$
f(t)=\frac {t^2}2\intr |\nabla
u_\ep|^2\,dx-\frac{t^{2_*}}{2^*}\intrr |u_\ep(x',0)|^{2_*}\,dx'
$$
attains its maximum at $t=1$ and the maximum is $\frac
1{2(N-1)}S_T^{N-1}$. Hence, to prove $I(t_\ep u_\ep)<0$ for $\ep$
small enough, it is sufficient to prove
\begin{equation}\label{4.16}
(N-2)^N\intrr \frac{g(x')}{[1+|x'|^2]^{N}}\,dx'-\frac
12(N-2)^{N}\intrr \frac{g(x')}{[1+|x'|^2]^{N-1}}\,dx'<0.
\end{equation}
In the following, we will show that equation \eqref{4.16} indeed
holds.

A direct calculation shows that
\begin{equation}\label{4.17}
\begin{split}
&\intrr \frac{g(x')}{[1+|x'|^2]^{N-1}}\,dx'\\&=\intrr
\frac{\sum_{i=1}^{N-1}\alpha_ix_i^2}{[1+|x'|^2]^{N-1}}\,dx'\\&=\frac
12H(0)\intrr \frac{|x'|^2}{[1+|x'|^2]^{N-1}}\,dx'\\&=\frac
12H(0)\omega_{N-2}\int_0^\infty \frac{r^N}{[1+r^2]^{N-1}}\,dr,
\end{split}
\end{equation}
where $H(0)=\frac 2{N-1}{\sum_{i=1}^{N-1}\alpha_i}$ is the mean
curvature of $\partial \Omega$ at $0$, $\omega_{N-2}$ is the area of
the unit sphere in $R^{N-1}$. Similarly, we have
\begin{equation}\label{4.18}
\intrr \frac{g(x')}{[1+|x'|^2]^{N}}\,dx'=\frac 12H(0)\omega_{N-2}
\int_0^\infty \frac{r^N}{[1+r^2]^{N}}\,dr.
\end{equation}
On the other hand, we note that
\begin{equation}\label{4.19}
\begin{split}
&\int_0^\infty
\frac{r^N}{[1+r^2]^{N-1}}\,dr\\&=\frac{2(N-1)}{N+1}\int_0^{\infty}\frac{r^{N+2}}{[1+r^2]^N}\,dr\\
&=\frac{2(N-1)}{N+1}[\int_0^{\infty}\frac{(1+r^2)r^{N}}{[1+r^2]^N}\,dr-\int_0^{\infty}\frac{r^{N}}{[1+r^2]^N}\,dr]\\
&=\frac{2(N-1)}{N+1}[\int_0^{\infty}\frac{r^{N}}{[1+r^2]^{N-1}}\,dr-\int_0^{\infty}\frac{r^{N}}{[1+r^2]^N}\,dr].
\end{split}
\end{equation}
So we get
$$
\int_0^{\infty}\frac{r^{N}}{[1+r^2]^{N-1}}\,dr=\frac{2(N-1)}{N-3}\int_0^{\infty}\frac{r^{N}}{[1+r^2]^N}\,dr.
$$
We infer from the above equations that
\begin{equation}\label{4.20}
\begin{split}
&(N-2)^N\intrr \frac{g(x')}{[1+|x'|^2]^{N}}\,dx'-\frac
12(N-2)^{N}\intrr
\frac{g(x')}{[1+|x'|^2]^{N-1}}\,dx'\\&=-(N-2)^NH(0)\omega_{N-2}
\frac{1}{N-3}\int_0^{\infty}\frac{r^{N}}{[1+r^2]^{N}}\,dr<0
\end{split}
\end{equation}
provided $H(0)>0$. This proved the case $N\geq 5$.

The case $N=4$ is similar. We only need to replace $O(\ep^2)$ by
$O(\ep^2 |\ln \ep|)$ in equation \eqref{4.15}, the rest of the proof
is the same as $N\geq 5$. We omit the details.

\end{proof}

Finally, we study the case $N=3$. In this case, we suppose the
principal curvatures of $\partial \Omega$ at $x_0 \in\partial
\Omega$ belong to interval $(2a,2A)$ for some $0<a\leq A<\infty$.
Then we have $a|x'|^2\leq h(x')\leq A |x'|^2$ for $x'\in
D(0,\delta)$. Moreover, we have the following estimates.
\begin{lemma}\label{t 4.3}
Suppose $N=3$ and $x_0$ as above, then we have the following estimates.\\
(i) $\into |\nabla u_\ep|^2\,dx\leq \intr |\nabla
u_\ep|^2\,dx-C\ep|\ln
\ep|+O(\ep)$;\\
(ii)$\intpo u_\ep^{4}\,dS=\int_{\mathbb R^2}u_\ep(x',0)^4\,dx'-C\ep+o(\ep)$.\\
(iii) $\into u_\ep^2\,dx'=O(\ep)$.
\end{lemma}
\begin{proof}
As for (i), we infer from the above setting that
\begin{equation}\label{4.21}
\begin{split}
&\into |\nabla u_\ep|^2\,dx=\intr |\nabla
u_\ep|^2\,dx-\int_{D(0,\delta)}dx'\int_0^{h(x')}|\nabla
u_\ep|^2\,dx_N+O(\ep)\\ &\leq \intr |\nabla
u_\ep|^2\,dx-C\int_{D(0,\delta)}dx'\int_0^{a|x'|^2}\frac{\ep}{[(\ep+x_N)^2+|x'|^2]^2}\,dx_N+O(\ep)\\&\leq
\intr |\nabla u_\ep|^2\,dx-
C\int_{D(0,\ep^{-1}\delta)}dx'\int_0^{\ep a|x'|^2}\frac
1{[(1+x_N)^2+|x'|^2]^2}\,dx_N+O(\ep)\\&=\intr |\nabla
u_\ep|^2\,dx-C\ep
\int_{D(0,\ep^{-1}\delta)}\frac{a|x'|^2}{[1+|x'|^2]^2} dx'+O(\ep)
\\&=\intr |\nabla
u_\ep|^2\,dx-C\ep[C+|\ln \ep|]+O(\ep)\\&=\intr |\nabla
u_\ep|^2\,dx-C\ep|\ln \ep|+O(\ep),
\end{split}
\end{equation}
this proves (i).

(ii) is the same as $N\geq 4$, we omit the details.

(iii) can be proved in a similar way as in \cite{BN}. In fact, we
have
\begin{equation}\label{4.22}
\begin{split}
\into u_\ep^2\,dx&\leq
C\int_{B_R}\frac{\ep}{(\ep+x_N)^2+|x'|^2}\,dx\\&\leq C\ep
\int_{B_R}\frac 1{\ep^2+|x|^2}\,dx\\&=C\ep
\int_0^R\frac{r^2}{\ep^2+r^2}\,dr\\&=C\ep[\int_0^{\ep}\frac
12\,dr+\int_\ep^R 1\,dr]\\&=C\ep[\frac 12\ep+(R-\ep)]
\\&\leq C\ep.
\end{split}
\end{equation}
This finishes the proof of this lemma.

\end{proof}
\begin{proof}[Proof of Theorem \ref{t 1.3} with $N=3$:]
By the same reason as for $N\geq 4$, we conclude that the functional
$I$ possesses the Mountain Pass structure, so we only need to show
that the Mountain Pass level is below $\frac 1{2(N-1)}S_T^{ N -1}$.
For this purpose, we let $u_\ep, t_\ep$ the same as the case $N\geq
4$, then we have
\begin{equation}\label{4.23}
\begin{split}
I(t_\ep u_\ep)&\leq \frac 1{2(N-1)}S_T^{N-1}-C\ep|\ln \ep|-C\into
u_\ep^r\,dr+O(\ep).
\end{split}
\end{equation}
It is easy to conclude from the above inequality that
$$
I(t_\ep u_\ep)<  \frac 1{2(N-1)}S_T^{N-1}
$$
for $\ep>0$ small enough. This finishes the proof of Theorem \ref{t
1.3} for $N=3$.
\end{proof}

\section{Proof of Theorem \ref{1.4}}
In the finally section, we study the existence result of equation
\eqref{1.1} with double critical exponents, i.e., $r=2^*$ and
$q=2_*$. We first consider the following limit problem
\begin{equation}\label{5.1}
  \left\{
  \begin{array}{ll}
  \displaystyle
\lp u=u^{2^*-1}   &{\rm in}\quad \mathbb R_+^N,    \\
\\ \frac{\partial u}{\partial \nu}=u^{2_*-1} &{\rm on }\quad  \partial
\mathbb R_+^N.
\end{array}
\right.
\end{equation}
Y.Li and M.Zhu \cite{LZu} classified all the positive solutions of
problem \eqref{5.1}. More precisely, they proved the positive
solutions have the following form
\begin{equation}\label{5.2}
u_\ep(x)=\frac{[N(N-2)]^{\frac{N-2}4}\ep^{\frac
{N-2}2}}{[\ep^2+|x'|^2+(x_N+\ep x_N^0)^2]^{\frac{N-2}2}},
\end{equation}
where $x_N^0=(\frac N{N-2})^{\frac 12}$. Let
$$
I_\infty(u)=\frac 12\int_{\R} |\nabla u|^2\,dx-\frac 1{2^*}\int_{\R}
|u|^{2^*}\,dx-\frac 1{2_*}\int_{\BR} |u(x,0)|^{2_*}\,dx'
$$
defined on $D^{1,2}(\R)$, $c_\infty=I_\infty(u_\ep)$ and
\begin{equation}\label{5.3}
I(u)=\frac 12\into |\nabla u|^2+u^2\,dx-\frac 1{2^*}\into
|u|^{2^*}\,dx-\frac 1{{2_*}}\intpo |u|^{2_*}\,dS
\end{equation}
defined on $H^1(\Omega)$, then we have the following local
compactness result.
\begin{lemma}\label{t 5.1}
Let $\{u_n\}$ be a $(PS)_c$ sequence for $I$ with $c\in
(0,c_\infty)$, that is, $I(u_n)\to c$ and $I'(u_n)\to 0$ as
$n\to\infty$, then $\{u_n\}$ has a convergent subsequence.
\end{lemma}
\begin{proof}
Let $\{u_n\}$ be a $(PS)_c$ sequence for $I$ with $c\in
(0,c_\infty)$, that is,
$$
I(u_n)\to c
$$
and
$$
I'(u_n)\to 0
$$
as $n\to \infty$, we will show that $\{u_n\}$ has a convergent
subsequence.

We first show that $\{u_n\}$ is bounded. By means of a $(PS)_c$
sequence, we have the following two equations
$$
I(u_n)=\frac 12\into |\nabla u_n|+u_n^2\,dx-\frac 1{2^*}\into
|u_n|^{2^*}\,dx-\frac 1{2_*}\intpo |u_n|^{2_*}\,dS=c+o(1)
$$
and
$$
\langle I'(u_n),u_n\rangle=\into |\nabla u_n|+u_n^2\,dx-\into
|u_n|^{2^*}\,dx-\intpo |u_n|^{2_*}\,dS=o(1)\|u_n\|.
$$
We infer from the above two equations that
$$
(\frac 12-\frac 1{2_*})\|u_n\|^2\leq I(u_n)-\frac 1{2_*}\langle
I'(u_n),u_n\rangle=c+o(1)+o(1)\|u_n\|.
$$
Hence, we conclude from the above equation that $\{u_n\}$ is
bounded.

Next, we show that $\{u_n\}$ has a convergent subsequence. Since
$u_n$ is bounded, we can suppose that $u_n\rightharpoonup u_0$ in
$H^1(\Omega)$, $u_n\rightharpoonup u_0$ in $L^{2^*}(\Omega)$,
$u_n\rightharpoonup u_0$ in $L^{2_*}(\partial \Omega)$ and $u_n\to
u_0$ in $L^2(\Omega)$.  Let $v_n=u_n-u_0$, if $v_n\to 0$ in
$H^1(\Omega)$, then the proof is complete. So in the following, we
assume that $v_n\not\to 0$. We deduce from Brezis-Lieb Theorem
\cite{BL} that
\begin{equation}\label{5.4}
\frac 12\into |\nabla v_n|\,dx-\frac 1{2^*}\into
|v_n|^{2^*}\,dx-\frac 1{2_*}\intpo |v_n|^{2_*}\,dS=c+o(1)-I(u_0)
\end{equation}
and
\begin{equation}\label{5.5}
\into |\nabla v_n|\,dx-\into |v_n|^{2^*}\,dx-\intpo
|v_n|^{2_*}\,dS=o(1).
\end{equation}
We distinguishes two cases:\\
 Case 1: $\intpo |v_n|^{2_*}\,dS\not\to 0$.\\
 Case 2: $\intpo |v_n|^{2_*}\,dS\to 0$ but $\into |v_n|^{2^*}\,dx\not\to
 0$.

If case 1 occurs, then we conclude that $\{v_n\}$ is a $(PS)$
sequence for $I_\infty$, hence we have that
$$
\lim\inf_{n\to \infty}I(v_n)\geq c_\infty,
$$
which further implies
$$
\lim\inf_{n\to \infty}I(u_n)=\lim\inf_{n\to \infty}I(v_n)+I(u_0)\geq
\lim\inf_{n\to \infty}I(v_n) \geq c_\infty,
$$
which contradicts that $c\in (0,c_\infty)$.

Similarly, if case 2 occurs, then we conclude that $\{v_n\}$ is a
$(PS)$ sequence for $\bar I$ or $\tilde I$ which are defined by
$$
\bar I(u)=\frac 12\intr |\nabla u|^2\,dx-\frac 1{2^*}\intr
|u_n|^{2^*}\,dx,\quad u\in D^{1,2}(\R)
$$
and
$$
\tilde I(u)=\frac 12\int_{\mathbb R^N} |\nabla u|^2\,dx-\frac
1{2^*}\int_{\mathbb R^N} |u_n|^{2^*}\,dx,\quad u\in D^{1,2}(\mathbb
R^N)
$$
respectively. We note that the ground state levels of $\bar I$ and
$\tilde I$ are $\frac{1}{2N}S^{\frac N2}$ and $\frac{1}{N}S^{\frac
N2}$ respectively. Moreover, a easy calculation shows that
$\frac{1}{2N}S^{\frac N2}>c_\infty$. So we deduce that
$$
\lim\inf_{n\to \infty}I(u_n)\geq \lim\inf_{n\to
\infty}I(v_n)+I(u_0)\geq \lim\inf_{n\to \infty}I(v_n) > c_\infty,
$$
which contradicts that $c\in (0,c_\infty)$.

Finally, we must have $v_n\to 0$ or $u_n\to u_0$ in $H^1(\Omega)$.
This completes the proof of this lemma.
\end{proof}

In order to imply the Mountain Pass theorem to functional $I$, we
must show that the Mountain Pass level of $I$ is indeed below
$c_\infty$. As before, since $\Omega$ is bounded, then exists a ball
$B_R(\bar x)$ containing $\Omega$ and $\partial B_R(\bar x)\cap \bar
\Omega\neq \emptyset$. Suppose that $x_0\in \partial B_R(\bar x)\cap
\bar \Omega$, then we have $2\alpha_i\geq \frac 1R$ for each $1\leq
i\leq N-1$, where $2\alpha_i(i=1,\cdots,N-1)$ are the principal
curvatures of $\partial \Omega$ at $x_0$. We can suppose $x_0=0$ and
$\Omega\subset \{x:x_N>0\}$ without loss of generality, then the
boundary of $ \Omega$ can be represented by
$$
x_N=h(x')=\Sigma_{i=1}^{N-1}\alpha_ix_i^2+o(|x'|^2),\quad \forall\
x'=(x_1,x_2,\cdots,x_{N-1})\in D(0,\delta)
$$
for some $\delta>0$, where $D(0,\delta)=B_\delta
(0)\cap\{x:x_N=0\}$.

With the above notations, we have the following estimates.
\begin{lemma}\label{t 5.2}
Let $u_\ep$ be defined by equation \eqref{5.2} and suppose $N\geq
4$, then we have\\
(i) $\int_\Omega|\nabla u_\ep|^2\,dx=\int_{\R}|\nabla
u_\ep|^2\,dx-N^{\frac{N-2}2}(N-2)^{\frac
{N+2}2}\ep\int_{R^{N-1}}\frac{[|x'|^2+(x_N^0)^2]g(x')}{[1+|x'|^2+(x_N^0)^2]^N}\,dx'+o(\ep)$.\\
(ii) $\int_\Omega u_\ep^{2^*}\,dx=\int_{\R}
u_\ep^{2^*}\,dx-N^{\frac{N}2}(N-2)^{\frac
{N}2}\ep\int_{R^{N-1}}\frac{g(x')}{[1+|x'|^2+(x_N^0)^2]^N}\,dx'+o(\ep)$.\\
(iii) $\int_{\partial\Omega}u_\ep^{2_*}\,dS=\int_{\mathbb R^{N-1}}
u_\ep^{2_*}\,dx'-2(N-1)N^{\frac{N}2}(N-2)^{\frac
{N-2}2}\ep\int_{R^{N-1}}\frac{g(x')}{[1+|x'|^2+(x_N^0)^2]^N}\,dx'+o(\ep).$\\
(iv)
$\into u_\ep^2\,dx=\left\{
  \begin{array}{ll}
  \displaystyle
 O(\ep^2|\ln \ep|) &\quad  N=4,\\
\\O(\ep^2)  &\quad N\geq 5,
\end{array}
\right.$
\end{lemma}
\begin{proof}
The proof of this lemma is similar to the previous ones. We sketch
it. For (i), a direct calculation shows that
\begin{equation}
\begin{split}
&-\int_{\Omega}|\nabla u_\ep|^2\,dx+\int_{\R}|\nabla
u_\ep|^2\,dx\\&=N^{\frac {N-2}2}(N-2)^{\frac {N+2}2}\ep^{N-2}\intrr
dx'\int_0^{g(x')} \frac{|x'|^2+(x_N+\ep
x_N^0)^2}{[\ep^2+|x'|^2+(x_N+\ep
x_N^0)^2]^N}\,dx_N+o(\ep)\\&=N^{\frac {N-2}2}(N-2)^{\frac
{N+2}2}\ep^{N-2}\intrr \ep^{N-1}\,dy'\int_0^{\ep
g(y')}\frac{\ep^2[|y'|^2+(y_N+x_N^0)^2]}{\ep^{2N}[1+|y'|^2+(y_N+x_N^0)^2]^N}\ep
\,dy_N+o(\ep)\\&=N^{\frac {N-2}2}(N-2)^{\frac {N+2}2} \ep\intrr
\frac{[|y'|^2+(x_N^0)^2]g(y')}{[1+|y'|^2+(x_N^0)^2]^N}\,dy'+o(\ep),
\end{split}
\end{equation}
hence (i) follows.

Similarly, as for (ii), we have
\begin{equation}
\begin{split}
&\intr u_\ep^{2^*}\,dx-\into u_\ep^{2^*}\,dx\\&=N^{\frac
{N}2}(N-2)^{\frac {N}2}\intrr
dx'\int_0^{g(x')}\frac{\ep^N}{[\ep^2+|x'|^2+(x_N+\ep
x_N^0)^2]^N}\,dx_N+o(\ep)\\&=N^{\frac {N}2}(N-2)^{\frac {N}2}\intrr
\ep^{N-1}dy'\int_0^{\ep
g(y')}\frac{\ep^N}{\ep^{2N}[1+|y'|^2+(y_N+x_N^0)^2]^N}\ep\,dy_N+o(\ep)\\&=N^{\frac
{N}2}(N-2)^{\frac {N}2}\ep
\intrr\frac{g(y')}{[1+|y'|^2+(x_N^0)^2]^N}\,dy'+o(\ep),
\end{split}
\end{equation}
(ii) follows.

As for (iii), we define
$$
f(r)=u_\ep(x',rg(x'))^{\frac{2(N-1)}{N-2}}\sqrt{1+r^2|\nabla
g(x')|^2},
$$
then we get
\begin{equation}
\begin{split}
&\intpo u_\ep^{\frac{2(N-1)}{N-2}}\,d\sigma-\intrr
u_\ep(x',0)^{\frac{2(N-1)}{N-2}}\,dx'\\ &=\intrr
u_\ep(x',g(x'))^{\frac{2(N-1)}{N-2}}\sqrt{1+|\nabla
g(x')|^2}\,dx'-\intrr
u_\ep(x',0)^{\frac{2(N-1)}{N-2}}\,dx'+o(\ep)\\&=\intrr
f(1)-f(0)\,dx'+o(\ep)\\&=-2(N-1)N^{\frac
{N-1}2}(N-2)^{\frac{N-1}2}\ep^{N-1}\\&\intrr
\frac{1}{[\ep^2+|x'|^2+(r_\ep g(x')+\ep x_N^0)^2]^N}[r_\ep g(x')+\ep
x_N^0]g(x')\sqrt{1+r_\ep^2|\nabla
g(x')|^2}\,dx'+o(\ep)\\&=-2(N-1)N^{\frac
{N-1}2}(N-2)^{\frac{N-1}2}\ep^{N-1}\\&\intrr \frac{\ep [r_\ep \ep
g(y')+x_N^0]\ep^2 g(y')}{\ep^{2N}[1+|y'|^2+(r_\ep \ep
g(y')+x_N^0)^2]^N}\sqrt{1+r_\ep^2\ep^2|\nabla
g(y')|^2}\ep^{N-1}\,dy'+o(\ep)
\\&=-2(N-1)N^{\frac {N-1}2}(N-2)^{\frac{N-1}2}\ep\intrr
\frac{x_N^0g(y')}{[1+|y'|^2+(x_N^0)^2]^N}dy'+o(\ep)\\&=-2(N-1)N^{\frac
N2}(N-2)^{\frac {N-2}2}\ep\intrr
\frac{g(y')}{[1+|y'|^2+(x_N^0)^2]^N}dy'+o(\ep),
\end{split}
\end{equation}
we get (iii).

(iv) is proved in \cite{BN}.

\end{proof}

\begin{lemma}\label{t 5.3}
Suppose that $N\geq 4$, the the following inequality holds
$$\sup_{t\geq 0}I(tu_\ep)<c_\infty.$$
\end{lemma}
\begin{proof}
We infer from Lemma \ref{t 5.2} that
\begin{equation}
\begin{split}
I(tu_\ep)&=\frac {t^2}2\intr |\nabla u_\ep|^2\,dx-\frac
{t^{2^*}}{2^*}\intr u_\ep^{2^*}\,dx-\frac{t^{2_*}}{2_*}\intrr
u_\ep(x',0)^{2_*}\,dx'\\&-\frac {t^2}2N^{\frac{N-2}2}(N-2)^{\frac
{N+2}2}\ep\intrr
\frac{[|x'|^2+(x_N^0)^2]g(x')}{[1+|x'|^2+(x_n^0)^2]^N}\,dx'\\&+\frac
{t^{2^*}}{2^*}N^{\frac{N}2}(N-2)^{\frac {N}2}\ep\intrr
\frac{g(x')}{[1+|x'|^2+(x_n^0)^2]^N}\,dx'\\&+\frac
{t^{2_*}}{2_*}2(N-1)N^{\frac{N}2}(N-2)^{\frac {N-2}2}\ep\intrr
\frac{g(x')}{[1+|x'|^2+(x_n^0)^2]^N}\,dx'+o(\ep).
\end{split}
\end{equation}
We assume that $I(tu_\ep)$ attains its maximum at $t_\ep$, then we
have $t_\ep=1+o(1)$ as $\ep\to 0$. On the other hand, we infer from
the definition of $c_\infty$ that
\begin{equation}
\begin{split}
&\frac {t_\ep^2}2\intr |\nabla u_\ep|^2\,dx-\frac
{t_\ep^{2^*}}{2^*}\intr
u_\ep^{2^*}\,dx-\frac{t_\ep^{2_*}}{2_*}\intrr
u_\ep(x',0)^{2_*}\,dx'\\&\leq \frac {1}2\intr |\nabla
u_\ep|^2\,dx-\frac {1}{2^*}\intr u_\ep^{2^*}\,dx-\frac{1}{2_*}\intrr
u_\ep(x',0)^{2_*}\,dx'\\&=c_\infty.
\end{split}
\end{equation}
Hence, to show $I(t_\ep u_\ep)<c_\infty$, it is sufficient to show
that

\begin{equation}\label{5.11}
\begin{split}
&-\frac 12N^{\frac{N-2}2}(N-2)^{\frac {N+2}2}\intrr
\frac{[|x'|^2+(x_N^0)^2]g(x')}{[1+|x'|^2+(x_N^0)^2]^N}\,dx'\\&+\frac
12N^{\frac{N-2}2}(N-2)^{\frac {N+2}2}\intrr
\frac{g(x')}{[1+|x'|^2+(x_n^0)^2]^N}\,dx'\\&+N^{\frac{N}2}(N-2)^{\frac
{N}2}\intrr \frac{g(x')}{[1+|x'|^2+(x_n^0)^2]^N}\,dx'<0.
\end{split}
\end{equation}
We need to compare the value
$$
\intrr \frac{|x'|^2g(x')}{[1+|x'|^2+(x_n^0)^2]^N}\,dx'
$$
with
$$
\intrr \frac{g(x')}{[1+|x'|^2+(x_n^0)^2]^N}\,dx'.
$$
If we denote $c=1+(x_N^0)^2=\frac{2(N-1)}{N-2}$, then a direct
calculation shows that
\begin{equation}
\begin{split}
&\intrr \frac{|x'|^2g(x')}{[1+|x'|^2+(x_n^0)^2]^N}\,dx'\\&=\intrr
\frac{|x'|^2(\Sigma_{i=1}^{N-1}\alpha_ix_i^2)}{[c+|x'|^2]^N}\,dx'\\&=(\Sigma_{i=1}^{N-1}\alpha_i)\intrr
\frac{|x'|^2x_i^2}{[c+|x'|^2]^N}\,dx'\\&=\frac 12H(0)\intrr
\frac{|x'|^4}{[c+|x'|^2]^N}\,dx'\\&=\frac
12H(0)\omega_{N-2}\int_0^{\infty}\frac{r^{N+2}}{[c+r^2]^N}\,dr.
\end{split}
\end{equation}
Similarly, we have
\begin{equation}
\begin{split}
&\intrr \frac{g(x')}{[1+|x'|^2+(x_n^0)^2]^N}\,dx'\\&=\intrr
\frac{(\Sigma_{i=1}^{N-1}\alpha_ix_i^2)}{[c+|x'|^2]^N}\,dx'\\&=(\Sigma_{i=1}^{N-1}\alpha_i)\intrr
\frac{x_i^2}{[c+|x'|^2]^N}\,dx'\\&=\frac 12H(0)\intrr
\frac{|x'|^2}{[c+|x'|^2]^N}\,dx'\\&=\frac
12H(0)\omega_{N-2}\int_0^{\infty}\frac{r^{N}}{[c+r^2]^N}\,dr.
\end{split}
\end{equation}
On the other hand, we have
\begin{equation*}
\begin{split}
\int_0^\infty \frac{r^{N+2}}{[c+r^2]^N}\,dr&=\int_0^\infty
\frac{(c+r^2)r^N}{[c+r^2]^N}\,dr-c \int_0^\infty
\frac{r^N}{[c+r^2]^N}\,dr\\&=\frac{2(N-1)}{N+1}\int_0^\infty
\frac{r^{N+2}}{[c+r^2]^N}\,dr-\frac{2(N-1)}{N-2}\int_0^\infty
\frac{r^N}{[c+r^2]^N}\,dr,
\end{split}
\end{equation*}
which implies
$$
\int_0^\infty
\frac{r^{N+2}}{[c+r^2]^N}\,dr=\frac{2(N-1)(N+1)}{(N-2)(N-3)}\int_0^\infty\frac{r^{N}}{[c+r^2]^N}\,dr.
$$
Inserting this equation into the left hand side of equation
\eqref{5.11}, then we get
\begin{equation}\label{3.5}
\begin{split}
&-\frac 12N^{\frac{N-2}2}(N-2)^{\frac {N+2}2}\intrr
\frac{[|x'|^2+(x_N^0)^2]g(x')}{[1+|x'|^2+(x_N^0)^2]^N}\,dx'\\&+\frac
12N^{\frac{N-2}2}(N-2)^{\frac {N+2}2}\intrr
\frac{g(x')}{[1+|x'|^2+(x_n^0)^2]^N}\,dx'\\&+N^{\frac{N}2}(N-2)^{\frac
{N}2}\intrr \frac{g(x')}{[1+|x'|^2+(x_n^0)^2]^N}\,dx'\\&=\frac
{\omega_{N-2}}2N^{\frac{N-2}2}(N-2)
^{\frac{N}2}H(0)[-\frac{(N-1)(N+1)}{N-3}-\frac N2+\frac{N-2}2+N
]\int_0^\infty \frac{r^N}{(c+r^2)^{N}}\,dr\\&=-2N^{\frac{N-2}2}(N-2)
^{\frac{N}2}\frac{N-1}{N-3}H(0)\omega_{N-2}\int_0^\infty
\frac{r^N}{(c+r^2)^{N}}\,dr\\&<0
\end{split}
\end{equation}
since $H(0)>0$. This proves equation \eqref{5.11}. Hence we deduce
that
$$
I(t_\ep u_\ep)<0
$$
for $\ep$ small enough.
\end{proof}

Finally, we study the case $N=3$. In this case, we suppose the
principal curvatures of $\partial \Omega$ at $x_0 \in\partial
\Omega$ belong to interval $(2a,2A)$ for some $0<a\leq A<\infty$.
Then we have $a|x'|^2\leq h(x')\leq A |x'|^2$ for $x'\in
D(0,\delta)$. With these notations, we have the following estimates.
\begin{lemma}\label{t 5.4}
Let $N=3$, $\Omega$ as above and $u_\ep$ be defined by equation
\eqref{5.2}, then we have \\
(i) $\int_\Omega|\nabla u_\ep|^2\,dx\leq \intr |\nabla
u_\ep|^2\,dx-C\ep |\ln \ep|+O(\ep)$.\\
(ii) $\int_\Omega u_\ep^6\,dx\geq \intr u_\ep^6\,dx-C\ep$.\\
(iii) $\intpo u_\ep^4\,d\sigma=\intrr u_\ep^4\,dx'-C\ep$.\\
(iv) $\into u_\ep^2=O(\ep)$.

\end{lemma}
\begin{proof}
As for (i), we have
\begin{equation}
\begin{split}
\into |\nabla u_\ep|^2\,dx&=\intr |\nabla
u_\ep|^2\,dx-\int_{D(0,\delta)}dx'\int_0^{h(x')}|\nabla
u_\ep|^2\,dx_N+O(\ep^{N-2})\\&\leq \intr |\nabla
u_\ep|^2\,dx-\int_{D(0,\delta)}dx'\int_0^{a|x'|^2}|\nabla
u_\ep|^2\,dx_N+O(\ep).
\end{split}
\end{equation}
For the second term on the right hand of the above equation, we have
\begin{equation}
\begin{split}
&\int_{D(0,\delta)}dx'\int_0^{a|x'|^2}|\nabla u_\ep|^2\,dx_N\\&\geq
N^{\frac 12}(N-2)^{\frac 52}\ep
\int_{D(0,\delta)}dx'\int_0^{a|x'|^2}\frac{|x'|^2+(x_N+\ep
x_N^0)^2}{[\ep^2+|x'|^2+(x_N+\ep x_N^0)^2]^3}dx_N\\&=C
\int_{D(0,\ep^{-1}\delta)}dx'\int_0^{a\ep |x'|^2} \frac{|x'|^2+(x_N+
x_N^0)^2}{[1+|x'|^2+(x_N+ x_N^0)^2]^3}dx_N\\&\geq C
\int_{D(0,\ep^{-1}\delta)} \frac{|x'|^2+(x_N^0)^2}{[1+|x'|^2+(x_N+
x_N^0)^2]^3}\ep|x'|^2\,dx'
\\&=C\ep\int_0^{\ep^{-1}\delta}\frac{r^2+(x_N^0)^2}{[1+r^2+(x_N^0)^2]^3}r^3dr\\&=
C\ep[\int_0^{\sqrt{1+(x_N^0)^2}}\frac{r^5}{[1+r^2+(x_N^0)^2]^3}dr
+\int_{\sqrt{1+(x_N^0)^2}}^{\ep^{-1}\delta}\frac{r^5}{[1+r^2+(x_N^0)^2]^3}dr\\&+\int_0^{\ep^{-1}\delta}
\frac{(x_N^0)^2r^3}{[1+r^2+(x_N^0)^2]^3}dr ]\\&\geq
C\ep[|\ln\ep|+C].
\end{split}
\end{equation}
Hence, (i) follows.

For (ii), we first note that
\begin{equation}
\into u_\ep^6=\intrr
u_\ep^6\,dx-\int_{D(0,\delta)}dx'\int_0^{h(x')}u_\ep^6\,dx_N+O(\ep^{3}).
\end{equation}
While a direct calculation shows that
\begin{equation}
\begin{split}
&\int_{D(0,\delta)}dx'\int_0^{h(x')}h_\ep^6\,dx_N\\&\leq
C\int_{D(0,\delta)}\int_0^{A|x'|^2}\frac{\ep^3}{[\ep^2+|x'|^2+(x_N+\ep
x_N^0)^2]^3}dx_N\\&=C\int_{D(0,\ep^{-1}\delta)}dx'\int_0^{\ep
A|x'|^2}\frac 1{[1+|x'|^2+(x_N+ x_N^0)^2]^3}\,dx_N\\&\leq C
\int_{D(0,\ep^{-1}\delta)}\frac{A\ep|x'|^2}{[c+|x'|^2]^3}\,dx'\\&=C\ep\int_0^{\ep^{-1}\delta}\frac{r^3}{[c+r^2]^3}\,dr\\&=C\ep,
\end{split}
\end{equation}
hence (ii) follows.

For (iii), if we define
$$
f(r)=u_\ep(x',rg(x'))^4\sqrt{1+r^2|\nabla g(x')|^2},
$$
then we have
\begin{equation}
\begin{split}
&\intpo u_\ep^4-\intrr u_\ep(x',0)^4\,dx'\\&=\intrr
f(1)-f(0)\,dx'+O(\ep^2)\\&=C\intrr u_\ep(x',r_\ep
g(x'))^3\frac{\partial u_\ep}{\partial x_N}(x',r_\ep
g(x'))g(x')\sqrt {1+r_\ep^2|\nabla g(x')|^2}\,dx'+o(\ep)\\&=-C\intrr
\frac{\ep^2}{[\ep^2+|x'|^2+(r_\ep g(x')+\ep x_N^0)^2]^3}(r_\ep
g(x')+\ep x_N^0)g(x')\sqrt {1+r_\ep^2|\nabla
g(x')|^2}\,dx'+o(\ep)\\&= -C\ep\intrr
\frac{g(x')}{[1+|x'|^2+(x_N^0)^2]^3}dx'+o(\ep)\\&=C\ep+o(\ep),
\end{split}
\end{equation}
hence (iii) follows.

(iv) is proved in \cite{BN}.
\end{proof}

With the above preparations, we can estimate the Mountain Pass level
of $I$ for $N=3$ now. More precisely, we have the following result.
\begin{lemma}\label{t 5.5}
Suppose that $N=3$, then we have
$$\sup_{t\geq 0}I(tu_\ep)<c_\infty$$
for $\ep$ small enough.
\end{lemma}
\begin{proof}
We infer from Lemma \ref{t 5.4} that
\begin{equation}\label{5.20}
\begin{split}
I(t u_\ep)&=\frac {t^2}2\intr |\nabla u_\ep|^2\,dx-\frac
{t^{2^*}}{2^*}\intr u_\ep^{2^*}\,dx-\frac{t^{2_*}}{2_*}\intrr
u_\ep^{2_*}\,dx'\\&-\frac {t^2}2C\ep|\ln
\ep|+\frac{t^{2^*}}{2^*}C\ep+\frac{t^{2_*}}{2_*}C\ep+O(\ep).
\end{split}
\end{equation}
Suppose that $I(tu_\ep)$ attains its maximum at $t_\ep$, then it is
easy to see that $t_\ep=1+o(1)$ as $\ep\to 0$. Moreover, we have
\begin{equation}
\begin{split}
&\frac {t_\ep^2}2\intr |\nabla u_\ep|^2\,dx-\frac
{t_\ep^{2^*}}{2^*}\intr
u_\ep^{2^*}\,dx-\frac{t_\ep^{2_*}}{2_*}\intrr
u_\ep(x',0)^{2_*}\,dx'\\&\leq \frac {1}2\intr |\nabla
u_\ep|^2\,dx-\frac {1}{2^*}\intr u_\ep^{2^*}\,dx-\frac{1}{2_*}\intrr
u_\ep(x',0)^{2_*}\,dx'\\&=c_\infty.
\end{split}
\end{equation}
Insert this equation into equation \eqref{5.20}, then we get
$$
I(t_\ep u_\ep)<c_\infty
$$
for $\ep$ small enough. So the conclusion of this lemma follows.

\end{proof}

\begin{proof}[Proof of Theorem \ref{t 1.4}]
The proof of Theorem \ref{t 1.4} is a direct consequence of the
Mountain Pass theorem. In fact, it is easy to see that the
functional $I$ possesses the Mountain Pass structure, so the
Mountain Pass value $c$ is well-defined. Moreover, it follows from
Lemma \ref{t 5.1} that $I$ satisfies the $(PS)_c$ condition for
$c<c_\infty$. Finally, we infer from Lemma \ref{t 5.3} and Lemma
\ref{t 5.5} that $c<c_\infty$. So $c$ is a nontrivial critical value
for functional $I$, that is, problem \eqref{1.1} possesses a
nontrivial solution under the assumptions of Theorem \ref{t 1.4}.

\end{proof}

\bigskip
{\bf Acknowledgement}: This work is supported by NSFC, No.11101291.

\end{document}